\documentclass[10pt]{amsart}
\usepackage{amssymb}
\usepackage{color}
\usepackage{graphicx}
\usepackage{float}
\usepackage[bottom]{footmisc}
\usepackage[colorlinks,hyperindex,linkcolor=blue]{hyperref}
\hypersetup{
 pdfauthor = {Alvaro del Pino Gomez},
 pdftitle= {Turbulisation of contact foliations},
 pdfsubject = {Differential geometry, foliation, contact topology, contact foliation, even--contact structure}}

\newcommand{\SE}{{\mathcal{E}}}

\newcommand{\SW}{{\mathcal{W}}}

\newcommand{\SX}{{\mathfrak{X}}}

\newcommand{\SF}{{\mathcal{F}}}

\newcommand{\ot}{{\operatorname{OT}}}
\newcommand{\leg}{{\operatorname{leg}}}
\newcommand{\std}{{\operatorname{std}}}
\newcommand{\Lutz}{{\operatorname{Lutz}}}

\newcommand{\SL}{{\mathcal{L}}}

\newcommand{\R}{{\mathbb{R}}}

\newcommand{\NS}{{\mathbb{S}}}
\newcommand{\D}{{\mathbb{D}}}

\newcommand{\Op}{{\mathcal{O}p}}

\newtheorem{lemma}{Lemma}

\newtheorem{theorem}[lemma]{Theorem}

\newtheorem{definition}[lemma]{Definition}

\theoremstyle{remark}
\newtheorem{remark}[lemma]{Remark}

\begin{document} 

\title{Some tight contact foliations can be approximated by overtwisted ones}

\date{\today}


\author{\'Alvaro del Pino}
\address{Utrecht University, Department of Mathematics, Budapestlaan 6, 3584 Utrecht, The Netherlands}
\email{a.delpinogomez@uu.nl}

\begin{abstract}
A contact foliation is a foliation endowed with a leafwise contact structure. In this remark we explain a turbulisation procedure that allows us to prove that tightness is not a homotopy invariant property for contact foliations.
\end{abstract}

\maketitle

\section{Statement of the results}

Let $M^{2n+1+q}$ be a closed smooth manifold. Let $\SF^{2n+1}$ be a smooth codimension--$q$ foliation on $M$. We say that $(M,\SF)$ can be endowed with the structure of a \emph{contact foliation} if there is a hyperplane field $\xi^{2n} \subset \SF$ such that, for every leaf $\SL$ of $\SF$, $(\SL,\xi|_{\SL})$ is a contact manifold.

In \cite[Theorem 1.1]{CPP} it was shown that $(M^4,\SF^3)$ admits a leafwise contact structure if there exists a $2$--plane field tangent to $\SF$. This was later extended in \cite{BEM} to foliations of any dimension, any codimension, and admitting a leafwise formal contact structure. In both cases, the foliations produced have all leaves overtwisted; therefore, the meaningful question is whether one can construct and classify contact foliations with tight leaves.

\textbf{Instability of tightness under homotopies.} If the foliation $\SF$ is fixed, the parametric Moser trick \cite[Lemma 2.8]{CPP} implies that any two homotopic contact foliations $(\SF,\xi_0)$ and $(\SF,\xi_1)$ are actually isotopic by a flow tangent to the leaves. In particular, if $\SF$ is fixed, tightness is preserved under homotopies. Our main result states that this is not the case anymore if $\SF$ is allowed to move:
\begin{theorem} \label{thm:main1}
Let $N$ be a closed orientable $3$--manifold. There is a path of contact foliations $(N \times \NS^1,\SF_s,\xi_s)$, $s\in [0,1]$, satisfying:
\begin{itemize}
\item the leaves of $(N \times \NS^1,\SF_0,\xi_0)$ are tight,
\item the leaves of $(N \times \NS^1,\SF_s,\xi_s)$ are overtwisted, for all $s>0$.
\end{itemize}
\end{theorem}

\textbf{Foliations transverse to even--contact structures.} Given a codimension--$1$ contact foliation $(M^{2n+2},\SF^{2n+1},\xi^{2n})$ and a line field $\SX$ transverse to $\SF$, it is immediate that the codimension--$1$ distribution $\SE = \xi \oplus \SX$ is maximally non--integrable. Such distributions are called \emph{even--contact structures}.

The \emph{kernel} or \emph{characteristic foliation} of $\SE$ is a line field $\SW \subset \SE$ uniquely defined by the expression $[\SW,\SE] \subset \SE$. Given an even--contact structure $\SE$, any codimension--1 foliation transverse to its kernel is imprinted with a leafwise contact structure. It is natural to study the moduli of contact foliations arising in this manner from $\SE$. Our second result states:
\begin{theorem} \label{thm:main2}
Let $N$ be a closed orientable $3$--manifold. There are foliations $\SF_0$ and $\SF_1$ and an even--contact structure $\SE$ such that:
\begin{itemize}
\item the leaves of $(N \times \NS^1,\SF_0,\xi_0 = \SE \cap \SF_0)$ are tight,
\item the leaves of $(N \times \NS^1,\SF_1,\xi_1 = \SE \cap \SF_1)$ are overtwisted.
\end{itemize}
\end{theorem}
\begin{proof}
During the proof of Theorem \ref{thm:main1} we shall see that the contact foliations $(\SF_s,\xi_s)$, $s\in[0,1]$, are imprinted by the same even--contact structure $\SE$.
\end{proof}
This result is in line with the theorem of McDuff \cite{McD} stating that even--contact structures satisfy the complete $h$--principle: one should expect this flexibility to manifest in other ways.

\textbf{Acknowledgements.} This note developed during a visit of the author to V. Ginzburg in UCSC and it was V. Ginzburg that posed the question of whether tightness could potentially be stable under deformations. In this occasion the question is certainly more clever than the small observation that provides the (negative) answer. The author is also grateful to F. Presas for reading this note and providing valuable suggestions. The author is supported by the grant NWO Vici Grant no. 639.033.312.

\section{Turbulisation of contact foliations}

We will now explain how to turbulise a contact foliation along a loop of legendrian knots. 

\subsection{Local model around a loop of legendrian knots}

Let $(N,\xi)$ be a contact $3$--manifold. Any legendrian knot $K \subset (N,\xi)$ has a tubular neighbourhood with the following normal form:
\[ (\Op(K) \subset N, \xi) \cong (\D^2 \times \NS^1, \xi_\leg = \ker(\cos(z)dx+\sin(z)dy)), \]
where $(x,y,z)$ are the coordinates in $\D^2 \times \NS^1$. A convenient way of thinking about the model is that it is simply the space of oriented contact elements of the disc. In particular, any diffeomorphism $\phi$ of $\D^2$ relative to the boundary induces a contactomorphism $C(\phi)$ of the model, also relative to the boundary, as follows:
\[ C(\phi)(x,y,z) = (\phi(x,y),d\phi(z)). \]
Here we think of $z$ as an oriented line in $T_{(x,y)}\D^2$ and we make $d\phi$ act by pushforward.

We can now define a contact foliation 
\[ (M_\leg = \D^2 \times \NS^1 \times \NS^1, \SF_\leg = \coprod_{t} \D^2 \times \NS^1 \times \{t\},  \xi_\leg), \]
which is simply the trivial bundle over $\NS^1$ with fibre the local model we just described. Given a contact foliation $(M,\SF,\xi)$ and an embedded torus $K: \NS^1\times\NS^1 \to M$ such that $K_t = K(t,-)$ is a legendrian knot on a leaf of $\SF$, it follows that there is an embedding $(M_\leg, \SF_\leg,  \xi_\leg) \to (M,\SF,\xi)$ providing a local model around $K$. It is sufficient to describe the turbulisation process in $(M_\leg, \SF_\leg, \xi_\leg)$.

\subsection{Fixing the even--contact structure}

Our aim now is to fix an even--contact structure $\SE_\leg$ in $(M_\leg, \SF_\leg)$ imprinting $\xi_\leg$. The reason why we do not simply choose $\xi_\leg \oplus \langle \partial_t\rangle$ is that $\SE_\leg$ should allow us to turbulise.

Take polar coordinates $(r,\theta)$ on $\D^2$. Construct a diffeomorphism $\phi: \D^2 \to \D^2$ such that:
\begin{itemize}
\item $\phi$ is of the form $\phi(r,\theta) = (f(r),\theta)$ for some function $f: [0,1] \to [0,1]$,
\item $f$ restricts to the identity in the complement of $[1/2,2/3]$,
\item $f$ compresses the interval $[1/2,2/3]$ towards the point $1/2$. 
\end{itemize}
See Figure \ref{fig:graph} for a depiction of the graph of $f$. Fix a vector field $h(r)\partial_r$, with $h(r) < 0$ in the region $r\in(1/2,2/3)$ and $h(r) = 0$ everywhere else, whose time--$1$ map is the function $f(r)$. There exists a unique vector field $X$ in $\D^2 \times \NS^1$ satisfying:
\begin{itemize}
\item $X$ is a contact vector field for the structure $\xi_\leg$,
\item $X$ is a lift of $h(r)\partial_r$. In particular, $X$ has a negative radial component in the region $r\in(1/2,2/3)$.
\end{itemize}
By construction the $3$--distribution $\SE_\leg(x,y,z,t) = \xi_\leg \oplus \langle \partial_t + X(x,y,z) \rangle$ is an even--contact structure whose kernel is $\SW_\leg = \langle \partial_t + X(x,y,z) \rangle$ and whose imprint on $(M_\leg, \SF_\leg)$ is precisely $\xi_\leg$.

\begin{center}
\begin{figure}[h!]
\centering
  \includegraphics[scale=0.4]{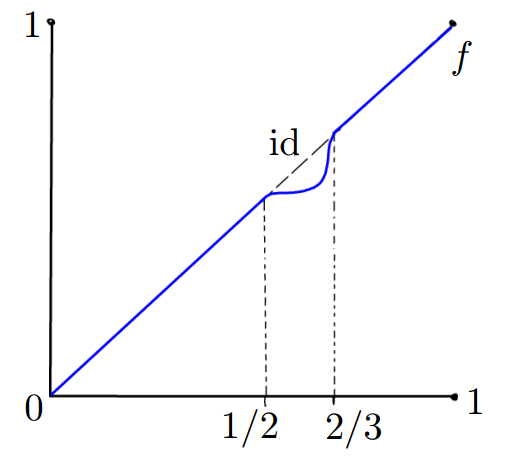}
  \caption{The function $f$.}
  \label{fig:graph}
\end{figure}
\end{center}

\subsection{Turbulisation}

Consider the surface $S = [0,1] \times \NS^1$ with coordinates $(r,t)$; $M_\leg$ projects onto $S$ in the obvious way. Under this projection the kernel $\SW_\leg$ is mapped to the  line field $L = \langle \partial_t + h(r)\partial_r \rangle$. Similarly, the foliation $\SF_\leg$ is simply the pullback of the line field $F_1 = \langle \partial_r \rangle$. $F_1$ and $L$ are transverse to one another. We can find a homotopy of line fields $(F_s)_{s\in[0,1]}$ in $S$ satisfying:
\begin{itemize}
\item $F_1 = \langle \partial_r \rangle$,
\item $F_s$ is transverse to $L$, for all $s$,
\item $F_s$ is isotopic to $F_1$ for every $s>0$,
\item $F_0$ is as in the last frame of Figure \ref{fig:foliation}: it has a closed orbit bounding a (half) Reeb component.
\end{itemize}
This path of line fields lifts to a path of codimension--$1$ foliations $\SF_{\leg,s}$ in $M_\leg$. $\SF_{\leg,1}$ is simply $\SF_\leg$ and $\SF_{\leg,s}$ is isotopic to it for every positive $s$. $\SF_{\leg,0}$ has a single compact leaf, which is diffeomorphic to $T^3$; this leaf bounds a Reeb component whose interior leaves are diffeomorphic to $\R^2 \times \NS^1$. Transversality of $L$ with respect to $F_s$ implies that $\SE_\leg$ imprints a contact foliation $\xi_{\leg,s}$ on each $\SF_{\leg,s}$.

\begin{lemma} \label{lem:tight}
The contact foliations in the homotopy $(M_\leg,\SF_{\leg,s},\xi_{\leg,s})$, $s\in[0,1]$, have all leaves tight.
\end{lemma}
\begin{proof}
The open leaves, as contact manifolds, are open subsets of the standard model $(\D^2 \times \NS^1, \xi_\leg)$, which is tight. The compact $T^3$ leaf is obtained by glueing the boundary components of a neighbourhood of the convex $T^2 = \partial (\D^2 \times \NS^1, \xi_\leg)$; it is contactomorphic to the space of oriented contact elements of the $2$--torus and therefore tight.
\end{proof}

Let us package this construction:
\begin{definition}
Let $(M,\SF,\xi)$ be a contact foliation. Suppose there is a region $U \subset M$ such that $(U,\SF,\xi)$ is diffeomorphic to the model $(M_\leg, \SF_\leg, \xi_\leg)$. We say that the homotopy $(M,\SF_s,\xi_s)_{s\in[0,1]}$ given by the procedure just described is the \textbf{turbulisation} of $(M,\SF,\xi)$ along $U$.
\end{definition}

\begin{remark}
There is an alternate way to describe the turbulisation process. $(M_\leg, \SF_\leg,  \xi_\leg)$ is simply the space of oriented contact elements of the foliation $(\D^2 \times \NS^1, \coprod_{t \in \NS^1}\D^2 \times \{t\})$: the foliation of the solid torus by its disc slices. Then, the turbulisation process upstairs amounts to turbulising $(\D^2 \times \NS^1, \coprod_{t \in \NS^1}\D^2 \times \{t\})$ and applying the contact elements construction. In particular, this highlights the fact that indeed the resulting leaves are tight. This construction also works for higher dimensional contact foliations.
\end{remark}

\begin{center}
\begin{figure}[h!]
\centering
  \includegraphics[scale=0.45]{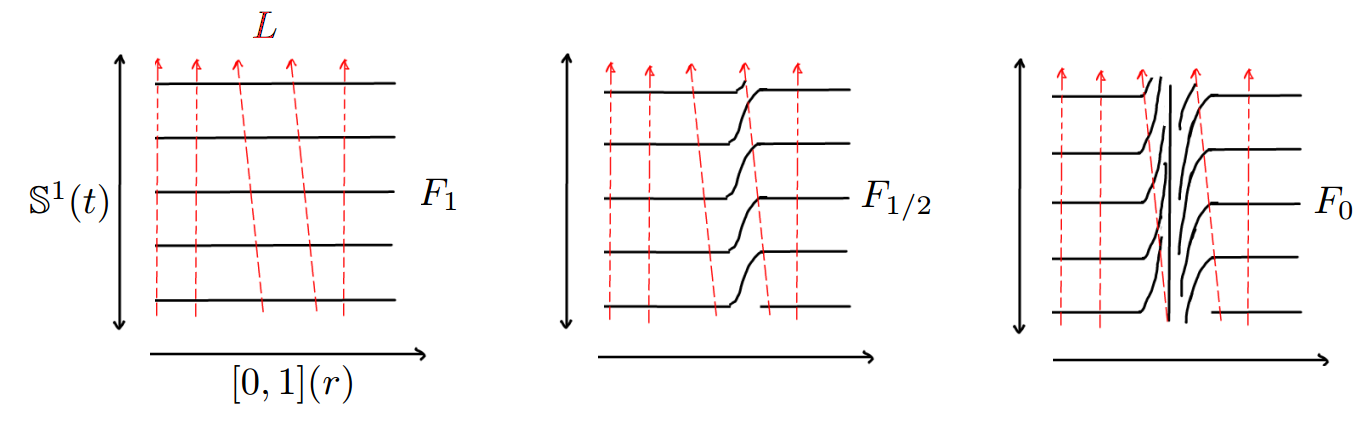}
  \caption{The solid lines represent (the foliations induced by) the path of line fields $(F_s)_{s\in[0,1]}$. The dotted ones with arrows on top represent the line field $L$.}
  \label{fig:foliation}
\end{figure}
\end{center}

\section{Applications}

\subsection{Proof of Theorems \ref{thm:main1} and \ref{thm:main2}}

In \cite{Dy} K. Dymara proved that there are legendrian links in overtwisted contact manifolds that intersect \emph{every} overtwisted disc; that is, their complement is tight. Such a link is said to be \emph{non--loose}. Let $(N,\xi)$ be an overtwisted contact manifold with $K$ a non--loose legendrian link. Consider the contact foliation
\[ (M,\SF_1,\xi_1) = (N \times \NS^1, \coprod_{t\in\NS^1} N \times \{t\}, \xi), \]
where we abuse notation and write $\xi$ for the leafwise contact structure lifting $(N,\xi)$. Take $U$ to be the tubular neighbourhood of $K \times \NS^1 \subset M$ and apply the turbulisation process to $(M,\SF_1,\xi_1)$ (on each component) to yield a path of contact foliations $(M,\SF_s,\xi_s)_{s\in[0,1]}$. It is immediate that $(M,\SF_s,\xi_s)$ is diffeomorphic to $(M,\SF_1,\xi_1)$ if $s$ is positive, because the foliations themselves are diffeomorphic and Gray's stability applies. In particular, the leaves of all of them are overtwisted. We claim that $(M,\SF_0,\xi_0)$ has all leaves tight. This is clear for the leaves in the Reeb components, as shown in Lemma \ref{lem:tight}. Similarly, the leaves outside of the Reeb components are tight because a neighbourhood of the non--loose legendrian link has been removed. We conclude by recalling that every closed overtwisted $3$--manifold admits a non--loose legendrian link: the legendrian push--off of the binding of a supporting open book \cite{EVV}. \hfill$\square$ 
 
\begin{remark}
The foliation $(M,\SF_1)$ is taut, since it admits a transverse $\NS^1$. As pointed out by V. Shende during a talk of the author: we are trading tautness of the foliation to achieve tightness of the leaves.
\end{remark}

\subsection{A more general statement}

A slightly more involved argument shows:

\begin{theorem} \label{thm:main3}
Let $M$ be a $4$--manifold. Suppose that $M$ admits a contact foliation $(\SF,\xi)$ with tight leaves. Then $M$ admits a contact foliation $(\SF_0,\xi_0)$ with tight leaves that can be approximated by contact foliations $(\SF_s,\xi_s)_{s\in(0,1]}$ with overtwisted leaves.
\end{theorem}
\begin{proof}
Find an embedded curve $\gamma: \NS^1 \to M$ transverse to $\SF$. This provides a $\NS^1$--family of Darboux balls $(\D^3,\xi_\std)$ along $\gamma$:
\[ (M_\std, \SF_\std, \xi_\std) = (\D^3 \times \NS^1, \coprod_{t \in\NS^1} \D^3 \times \{t\}, \xi_\std) \to (M,\SF,\xi)  \] 
Choose a legendrian knot $K \subset (\D^3,\xi_\std)$ and lift it to $K \times \NS^1 \subset (M_\std, \SF_\std, \xi_\std) \subset (M,\SF,\xi)$. Turbulisation in a neighbourhood of $K \times \NS^1$ yields a contact foliation $(M,\SF',\xi')$. The leaves of $(M,\SF',\xi')$ are still tight.

The interior of the Reeb component we just inserted is diffeomorphic, as a contact foliation, to the model $(M_\leg, \SF_\leg, \xi_\leg)$. Given a homotopically essential transverse knot $\eta \subset (\D^2\times\NS^1,\xi_\leg)$ we may perform a Lutz twist along $\eta$ to yield an overtwisted contact structure $\xi_\ot$ in $\D^2\times\NS^1$. The resulting local model along $\eta$ reads:
\[ (\D^2 \times \NS^1,\xi_\Lutz = \ker(f(r)dz + g(r)d\theta)) \]
where $(r,\theta,z)$ are the coordinates in a neighbourhood of $\eta = \{r=0\}$ and $r \to (f(r),g(r))/|f,g|$ is an immersion of $[0,1]$ onto $\NS^1$ that is injective for $r \in [0,1-\delta)$ and satisfies:
\begin{align*}
(f(r),g(r)) = (1,r^2) &\quad \textrm{ if } r \in [0,\delta] \\
f(r) = 0 &\quad \textrm{ if } r \in \{1/4,3/4\} \\
g(r) = 0 &\quad \textrm{ if } r \in \{0,1/2,1-\delta\} \\
(f(r),g(r)) = (1,(r-1+\delta)^2) &\quad \textrm{ if } r \in [1-\delta/2,1].
\end{align*}
The Lutz twist can be introduced parametrically \cite{CPP} to replace $(M_\leg, \SF_\leg, \xi_\leg) \subset (M,\SF',\xi')$ by $(M_\leg, \SF_\leg, \xi_\ot)$ in a $t$--invariant fashion. This produces a new contact foliation $(M,\SF_1,\xi_1)$ from $(M,\SF',\xi')$.

Set $K'(z) = (1/4,0,z) \in (\D^2 \times \NS^1,\xi_\Lutz) \subset (\D^2\times\NS^1,\xi_\ot)$; we shall prove that it is non--loose. The quasi--prelagrangian tori 
\[ \{r=r_0>1/4\} \subset (\D^2 \times \NS^1,\xi_\Lutz) \subset (\D^2\times\NS^1,\xi_\ot) \]
are incompressible in  $(\D^2\times\NS^1,\xi_\ot) \setminus K'$ due to our choice of $\eta$ and $K'$. We invoke \cite[Th\'eor\`eme 4.2]{Co}: $(\D^2\times\NS^1,\xi_\ot) \setminus K'$ is universally tight if and only if it is universally tight after removing any finite collection of such tori. Choose the tori at radii $r=1/2,1-\delta$. The reader can check that the pieces $\{r <1/2\}$, $\{1/2<r<1-\delta\}$ have standard tight $\R^3$ as their universal cover. The remaining piece, which intersects $(\D^2 \times \NS^1,\xi_\Lutz)$ in $\{r>1-\delta\}$, is contactomorphic to the complement of $\eta$ in $(\D^2\times\NS^1,\xi_\leg)$ and is therefore tight as well.

We turbulise in a neighbourhood of 
\[ K' \times \NS^1 \subset (M_\leg, \SF_\leg, \xi_\ot) \subset (M,\SF_1,\xi_1) \]
to produce the claimed family $(M,\SF_s,\xi_s)_{s\in[0,1]}$ and conclude the proof.
\end{proof}
The reader can check that the resulting foliation $(M,\SF_0,\xi_0)$ is in the same formal class as $(M,\SF,\xi)$, since $\SF_0$ is obtained from $\SF$ by turbulising twice and the even--contact structures inducing $\xi$ and $\xi_0$ differ from one another by a parametric (full) Lutz--twist.

A natural question to pose in light of Theorem \ref{thm:main3} is whether any $M^4$ admitting a formal contact foliation admits a foliation with tight leaves; the fundamental geometric issue towards achieving this is that it seems extremely delicate to ensure that no overtwisted disc is really present. For Theorem \ref{thm:main3} the main idea was to introduce the overtwisted discs in a controlled fashion so that they could later be destroyed.

\end{document}